\documentclass[a4paper]{amsart}

\usepackage[utf8]{inputenc}
\usepackage[english]{babel}

\usepackage[T1]{fontenc}
\usepackage{lmodern}  

\usepackage{my-math-symbol}
\usepackage{my-cleveref}
\usepackage{my-equation-numbering}
\usepackage{comment}
\usepackage{tikz}
\usetikzlibrary{cd}

\DeclareMathOperator{\pr}{pr}
\newcommand{\FS}{\omega_{\mathrm{FS}}}
\newcommand{\relFS}{\omega_{\mathrm{FS}}^{\mathrm{rel}}}

\usepackage[non-sorted-cites,initials,alphabetic,nobysame]{amsrefs}

\AtBeginDocument{%
	\def\MR#1{}
}

\title{Chern classes of Levi non-degenerate CR manifolds}
\author{Yuya Takeuchi}
\address{Division of Mathematics \\ Institute of Pure and Applied Sciences \\ University of Tsukuba
	\\ 1- 1- 1 Tennodai, Tsukuba, Ibaraki 305-8571 Japan}
\email{ytakeuchi@math.tsukuba.ac.jp, yuya.takeuchi.math@gmail.com}



\thanks{This work was supported by JSPS KAKENHI Grant Number JP25K17247.}

\begin{document}

\begin{abstract}
	We construct closed Levi non-degenerate CR manifolds with non-vanishing Chern monomials.
	For every $n \geq 2$ and $\lceil n / 4 \rceil \leq m \leq \lfloor n / 2 \rfloor$,
	there exists a $(2 n + 1)$-dimensional example of signature
	$(\lfloor n / 2 \rfloor, \lceil n / 2 \rceil)$
	on which all Chern monomials with positive indices at most $m$
	and total weight at most $n - m$ are non-zero in cohomology with real coefficients.
	Taking $m = \lceil n / 4 \rceil$ gives examples with
	$c_{1}^{\lfloor 3 n / 4 \rfloor} \neq 0$.
	For $n = 4$ or $n \geq 6$,
	the latter examples show that the Chern-class vanishing theorem for strictly pseudoconvex CR manifolds
	does not extend to the Levi non-degenerate case.
	In these dimensions,
	their underlying manifolds admit no strictly pseudoconvex CR structures.
	The construction uses circle bundles over products of integral pseudo-\Kahler manifolds.
\end{abstract}

\maketitle

\section{Introduction}
\label{section:introduction}

Let $(M, T^{1, 0} M)$ be a closed strictly pseudoconvex CR manifold of dimension $2 n + 1 \geq 5$.
It is known that
$M$ can be realized as the boundary of a strictly pseudoconvex domain
in a complex manifold~\cites{Boutet_de_Monvel1975,Harvey-Lawson1975,Lempert1995}.
This imposes restrictions on its topology.
Bungart~\cite{Bungart1992} and Popescu-Pampu~\cite{Popescu-Pampu2008} proved that
the cup product
\begin{equation}
	H^{i_{1}}(M; \bbR) \otimes \dots \otimes H^{i_{k}}(M; \bbR)
	\longrightarrow H^{i_{1} + \dots + i_{k}}(M; \bbR)
\end{equation}
vanishes if $i_{l} \leq n - 1$ and $i_{1} + \dots + i_{k} \geq n + 2$.
Moreover,
the author~\cite{Takeuchi2020-Chern} proved that
\begin{equation}
	c_{r_{1}}(T^{1, 0} M) \cup \dotsm \cup c_{r_{k}}(T^{1, 0} M) = 0
\end{equation}
in $H^{2 (r_{1} + \dots + r_{k})}(M; \bbR)$
if $2 (r_{1} + \dots + r_{k}) \geq n + 2$.
Note that these results are trivial when $n = 1$.

It is natural to ask whether these vanishing results hold for Levi non-degenerate CR manifolds.
In this paper,
we construct examples with $c_{1}^{\lfloor 3 n / 4 \rfloor} \neq 0$ in cohomology with real coefficients,
which gives counterexamples for $n = 4$ or $n \geq 6$.
More generally,
our main result gives simultaneous non-vanishing of Chern monomials on a single CR manifold.
For a multi-index $r = (r_{1}, \dots , r_{k})$ of positive integers,
we call $\abs{r} \coloneqq r_{1} + \dots + r_{k}$ its total weight.

\begin{theorem}
\label{thm:non-trivial-higher-Chern-class}
	Fix $n \geq 2$ and $\lceil n / 4 \rceil \leq m \leq \lfloor n / 2 \rfloor$.
	Then there exists a $(2 n + 1)$-dimensional closed Levi non-degenerate CR manifold $(S, T^{1, 0} S)$
	of signature $(\lfloor n / 2 \rfloor, \lceil n / 2 \rceil)$ such that
	\begin{equation}
		c_{r_{1}}(T^{1, 0} S) \cup \dotsm \cup c_{r_{k}}(T^{1, 0} S) \neq 0
	\end{equation}
	in $H^{2 \abs{r}}(S; \bbR)$
	for any integer $k \geq 1$ and any $(r_{1}, \dots , r_{k}) \in \{1, \dots , m\}^{k}$
	satisfying $\abs{r} \leq n - m$.
\end{theorem}

Applying \cref{thm:non-trivial-higher-Chern-class}
with $m = \lceil n / 4 \rceil$ and $r_{1} = \dots = r_{n - m} = 1$ gives the following consequence.

\begin{corollary}
\label{thm:non-trivial-first-Chern-class}
	For each $n \geq 2$,
	there exists a $(2 n + 1)$-dimensional closed Levi non-degenerate CR manifold $(S, T^{1, 0} S)$
	of signature $(\lfloor n / 2 \rfloor, \lceil n / 2 \rceil)$ such that
	\begin{equation}
		c_{1}(T^{1, 0} S)^{\lfloor 3 n / 4 \rfloor} \neq 0
	\end{equation}
	in $H^{2 \lfloor 3 n / 4 \rfloor}(S; \bbR)$.
	In particular,
	the cup product
	\begin{equation}
		H^{2}(S; \bbR)^{\otimes \lfloor 3 n / 4 \rfloor} \longrightarrow H^{2 \lfloor 3 n / 4 \rfloor}(S; \bbR)
	\end{equation}
	is non-trivial.
\end{corollary}

For $n = 4$ or $n \geq 6$,
we have $2 \lfloor 3 n / 4 \rfloor \geq n + 2$.
Thus \cref{thm:non-trivial-first-Chern-class} shows that the above Chern-class vanishing statement fails
when strict pseudoconvexity is replaced by Levi non-degeneracy.
Moreover,
the non-trivial cup product in that corollary depends only on the underlying smooth manifold.
Since $2 \leq n - 1$ in these dimensions,
the cup-product vanishing theorem implies that this manifold admits no strictly pseudoconvex CR structures,
or equivalently, no holomorphically fillable contact structures.

Our examples are circle bundles over closed integral pseudo-\Kahler manifolds $(Y, \omega)$.
The CR structure of such a circle bundle is isomorphic,
as a complex vector bundle,
to the pullback of the holomorphic tangent bundle of its base.
We use the Gysin exact sequence to detect non-vanishing Chern monomials on the total space.
More precisely,
we construct a closed real $(m, m)$-form $\alpha$ on the base
whose cohomology class is annihilated by $[\omega]$
and pairs non-trivially with the relevant Chern monomials
multiplied by a suitable power of the first Chern class.
The required bases are products of copies of three building blocks of complex dimensions $2$, $3$, and $4$.
The three- and four-dimensional blocks arise from
projective bundles over a complex torus and a nilmanifold,
respectively.

The paper is organized as follows.
In \cref{section:circle-bundles},
we recall the circle-bundle construction and formulate a non-vanishing criterion using the Gysin exact sequence.
In \cref{section:projective-bundles},
we develop the projective-bundle construction used in our examples.
In \cref{section:building-blocks},
we construct the three low-dimensional building blocks.
Finally,
in \cref{section:Chern-classes},
we combine these blocks and prove the main theorem.

Unless otherwise stated, Chern classes and cup products are considered in cohomology with real coefficients.
For a closed real two-form $\omega$,
the notation $[\omega] \in H^{2}(Y; \bbZ)$ means that
its de Rham class lies in the image of the natural map $H^{2}(Y; \bbZ) \to H^{2}(Y; \bbR)$.

\section{Circle bundles over pseudo-\Kahler manifolds}
\label{section:circle-bundles}

Let $M$ be an orientable smooth $(2 n + 1)$-dimensional manifold without boundary.
A \emph{CR structure} on $M$ is a rank $n$ complex subbundle $T^{1, 0} M$
of the complexified tangent bundle $T M \otimes \bbC$ such that
\begin{equation}
	T^{1, 0} M \cap T^{0, 1} M = 0, \qquad
	\comm{\Gamma(T^{1 ,0} M)}{\Gamma(T^{1, 0} M)} \subset \Gamma(T^{1, 0} M),
\end{equation}
where $T^{0, 1} M$ is the complex conjugate of $T^{1, 0} M$ in $T M \otimes \bbC$.
A typical example of a CR manifold is a real hypersurface $M$
in an $(n + 1)$-dimensional complex manifold $X$;
$M$ carries the canonical CR structure
\begin{equation}
	T^{1, 0} M
	\coloneqq T^{1, 0} X |_{M} \cap (T M \otimes \bbC).
\end{equation}

Let $\theta$ be a nowhere-vanishing real one-form on $M$
that annihilates $T^{1, 0} M$.
The \emph{Levi form} $\calL_{\theta}$ with respect to $\theta$ is the Hermitian form
on $T^{1,0} M$ defined by
\begin{equation}
	\calL_{\theta}(Z, W)
	\coloneqq - \sqrt{- 1} \, d \theta(Z, \ovW), \qquad Z, W \in T^{1, 0} M.
\end{equation}
A CR structure $T^{1, 0} M$ is said to be \emph{Levi non-degenerate}
if the Levi form is non-degenerate for some $\theta$;
such a $\theta$ is called a \emph{contact form}.
For a fixed contact form $\theta$,
the \emph{signature} of the Levi form at a point is the pair $(p, q)$,
where $p$ and $q$ are the numbers of positive and negative eigenvalues,
respectively.
Non-degeneracy implies that this pair is constant on each connected component of $M$.
For any positive smooth function $f$,
we have $\calL_{f \theta} = f \calL_{\theta}$,
so the signature is unchanged under positive rescaling of $\theta$.
Replacing $\theta$ by $- \theta$ interchanges $p$ and $q$.
Throughout this paper,
we say that $(M, T^{1, 0} M)$ has signature $(p, q)$
if it admits a contact form whose Levi form has
$p$ positive and $q$ negative eigenvalues at every point.

An important example of a Levi non-degenerate CR manifold
is the circle bundle associated with an integral pseudo-\Kahler manifold.
The underlying contact structure is obtained by the Boothby--Wang construction~\cite{Boothby-Wang1958};
see also \cite{Geiges2008}*{Section 7.2}.
We recall its realization using a Hermitian holomorphic line bundle,
which also equips the circle bundle with a natural CR structure.
Let $Y$ be an $n$-dimensional complex manifold
and let $\omega$ be a pseudo-\Kahler form on $Y$ of signature $(p, q)$ with $[\omega] \in H^{2}(Y; \bbZ)$.
Then there exists a Hermitian holomorphic line bundle $(L, h)$ over $Y$
such that
\begin{equation}
	\omega = - \frac{\sqrt{- 1}}{2 \pi} \Theta_{h} = \frac{1}{2 \pi} d d^{c} \log h,
\end{equation}
where $d^{c} \coloneqq (\sqrt{- 1} / 2)(\delb - \partial)$.
Consider the circle bundle
\begin{equation}
	S = \Set{ v \in L | h(v, v) = 1}
\end{equation}
over $Y$.
This is a real hypersurface in the total space of $L$
and has the canonical CR structure $T^{1, 0} S$.
The one-form $\eta \coloneqq d^{c} \log h |_{S}$
is a connection one-form on the principal $S^{1}$-bundle $p \colon S \to Y$
and satisfies $d \eta = 2 \pi p^{\ast} \omega$.
Moreover,
the natural CR structure $T^{1, 0} S$ coincides with
the horizontal lift of $T^{1, 0} Y$ with respect to $\eta$.
It follows from the definition of $\eta$ that
\begin{equation}
	\calL_{\eta}(Z, W)
	= - \sqrt{- 1} \, d \eta(Z, \ovW)
	= - 2 \pi \sqrt{- 1} \omega (p_{\ast} Z, p_{\ast} \ovW).
\end{equation}
This implies that $\calL_{\eta}$ is non-degenerate and of signature $(p, q)$.
In particular,
$(S, T^{1, 0} S)$ is a Levi non-degenerate CR manifold of signature $(p, q)$,
and $\eta$ is a contact form on $S$.

We now assume that $Y$ is closed.
The identification with the horizontal lift gives an isomorphism of complex vector bundles
\begin{equation}
	T^{1, 0} S \cong p^{\ast} T^{1, 0} Y,
\end{equation}
which implies
\begin{equation}
	c(T^{1, 0} S) = p^{\ast} c(T^{1, 0} Y).
\end{equation}
The Gysin exact sequence contains the segment
\begin{equation}
\label{eq:Gysin-exact-sequence}
	H^{2 k}(Y; \bbR) \xrightarrow{[\omega] \cup} H^{2 k + 2}(Y; \bbR)
	\xrightarrow{p^{\ast}} H^{2 k + 2}(S; \bbR).
\end{equation}
Consequently,
a class in $H^{2 k + 2}(Y; \bbR)$ has a non-zero pullback to $S$
if and only if it does not lie in the image of multiplication by $[\omega]$.
We will use the following sufficient condition.

\begin{lemma}
\label{lem:chern-non-vanishing-criterion}
	Let $\gamma \in H^{2 K}(Y; \bbR)$, where $1 \leq K \leq n$.
	Assume that there exists a class $\beta \in H^{2 n - 2 K}(Y; \bbR)$ such that
	\begin{equation}
		[\omega] \cup \beta = 0,
		\qquad
		\int_{Y} \beta \cup \gamma \neq 0.
	\end{equation}
	Then $p^{\ast} \gamma \neq 0$ in $H^{2 K}(S; \bbR)$.
\end{lemma}

\begin{proof}
	Suppose to the contrary that $p^{\ast} \gamma = 0$.
	Then \eqref{eq:Gysin-exact-sequence} gives
	$\gamma = [\omega] \cup \gamma^{\prime}$ for some $\gamma^{\prime} \in H^{2 K - 2}(Y; \bbR)$.
	Hence
	\begin{equation}
		\int_{Y} \beta \cup \gamma
		= \int_{Y} \beta \cup [\omega] \cup \gamma^{\prime}
		= 0,
	\end{equation}
	which is a contradiction.
\end{proof}

\section{Pseudo-\Kahler structures on projective bundles}
\label{section:projective-bundles}

In this section,
we give a projective-bundle construction that produces the three- and four-dimensional building blocks.
The key point is to obtain both an integral pseudo-\Kahler form
and a cohomology class that is annihilated by cup product with the class of this form
but has a non-zero Chern-class pairing.

\begin{proposition}
\label{prop:projective-bundle}
	Let $B$ be a closed complex manifold of dimension $k \geq 1$,
	and let $\omega_{0}$ and $\omega_{1}$ be closed real $(1, 1)$-forms on $B$.
	Assume that $c(T^{1, 0} B) = 1$,
	$[\omega_{0}], [\omega_{1}] \in H^{2}(B; \bbZ)$,
	$[\omega_{0}] \cup [\omega_{1}] = 0$,
	and
	\begin{equation}
		\omega_{t} \coloneqq (1 - t) \omega_{0} + t \omega_{1}
	\end{equation}
	is a pseudo-\Kahler form of signature $(p, q)$ for every $t \in \clcl{0}{1}$.
	Then there exist a closed pseudo-\Kahler manifold $(Y, \omega)$ of signature $(p, q + 1)$
	and a closed real $(1, 1)$-form $\alpha$ on $Y$ such that
	$c(T^{1, 0} Y) = 1 + c_{1}(T^{1, 0} Y)$,
	$[\omega] \in H^{2}(Y; \bbZ)$,
	$[\omega] \cup [\alpha] = 0$,
	and
	\begin{equation}
		\int_{Y} [\alpha] \cup c_{1}(T^{1, 0} Y)^{k} \neq 0.
	\end{equation}
	In particular,
	\begin{equation}
		c_{1}(T^{1, 0} Y)^{k} \notin \Im \rbra*{[\omega] \cup \colon H^{2 k - 2}(Y; \bbR) \longrightarrow H^{2 k}(Y; \bbR)}.
	\end{equation}
\end{proposition}

\begin{proof}
	Take a Hermitian holomorphic line bundle $(L, h_{L})$ over $B$ such that
	\begin{equation}
		\frac{\sqrt{- 1}}{2 \pi} \Theta_{h_{L}} = \omega_{1} - \omega_{0}.
	\end{equation}
	Let $\pi \colon Y \coloneqq \bbP(\bbC \oplus L) \to B$ be the projective bundle of $\bbC \oplus L$
	and let $S \coloneqq \calO_{Y}(- 1) \subset \pi^{\ast} (\bbC \oplus L)$ be
	the relative tautological line bundle on $Y$.
	Equip $\bbC \oplus L$ with the orthogonal sum metric
	and $S$ with the induced metric $h_{S}$.
	Consider the closed real $(1, 1)$-form
	\begin{equation}
		\omega
		\coloneqq \frac{\sqrt{- 1}}{2 \pi} \Theta_{h_{S}} + \pi^{\ast} \omega_{0},
	\end{equation}
	which defines an integral cohomology class.
	We first show that $\omega$ is a pseudo-\Kahler form of signature $(p, q + 1)$.
	Choose a local holomorphic frame $s$ of $L$ and write $h \coloneqq h_{L}(s, s)$.
	We use the affine fiber coordinate $w$ on the chart where the tautological line bundle $S$ is generated by
	\begin{equation}
		\tau \coloneqq e_{0} + w s,
	\end{equation}
	where $e_{0}$ is the standard unit vector of $\bbC$.
	Note that
	\begin{equation}
		h_{S}(\tau, \tau)
		= 1 + h_{L}(w s, w s)
		= 1 + h \abs{w}^{2}.
	\end{equation}
	Define
	\begin{equation}
		\eta_{w} \coloneqq d w + w h^{- 1} \partial h,
		\qquad
		t \coloneqq \frac{h \abs{w}^{2}}{1 + h \abs{w}^{2}}.
	\end{equation}
	A calculation gives
	\begin{align}
		\frac{\sqrt{- 1}}{2 \pi} \Theta_{h_{S}}
		&= - \frac{\sqrt{- 1}}{2 \pi} \partial \delb \log (1 + h \abs{w}^{2}) \\
		&= - \frac{\sqrt{- 1}}{2 \pi} \frac{h}{(1 + h \abs{w}^{2})^{2}} \eta_{w} \wedge \ovxet_{w}
			- \frac{\sqrt{- 1}}{2 \pi} t \partial \delb \log h \\
		&= - \relFS + t \pi^{\ast} (\omega_{1} - \omega_{0}),
	\end{align}
	where $\relFS$ is the relative Fubini-Study form on $Y$.
	Its restriction to each fiber is the Fubini-Study form.
	This implies
	\begin{equation}
		\omega
		= - \relFS + t \pi^{\ast} (\omega_{1} - \omega_{0}) + \pi^{\ast} \omega_{0}
		= - \relFS + \pi^{\ast} \omega_{t}.
	\end{equation}
	To check non-degeneracy along the infinity section, use the affine coordinate $v = 1 / w$
	and set $\eta_{v} \coloneqq d v - v h^{- 1} \partial h$.
	Since $\eta_{w} = - v^{- 2} \eta_{v}$ on the overlap, we have
	\begin{equation}
		t = \frac{h}{h + \abs{v}^{2}},
		\qquad
		\relFS = \frac{\sqrt{- 1}}{2 \pi} \frac{h}{(h + \abs{v}^{2})^{2}}
			\eta_{v} \wedge \ovxet_{v}.
	\end{equation}
	These expressions extend smoothly across $v = 0$.
	At each point of the infinity section,
	\begin{equation}
		\omega
		= - \frac{\sqrt{- 1}}{2 \pi h} d v \wedge d \ovv + \pi^{\ast} \omega_{1}.
	\end{equation}
	Thus $\omega$ has signature $(p, q + 1)$ there as well
	and is therefore a pseudo-\Kahler form of this signature on all of $Y$.
	
	Set $Q \coloneqq \pi^{\ast} (\bbC \oplus L) / S$,
	which is a holomorphic line bundle over $Y$.
	Consider the tautological exact sequence
	\begin{equation}
		0 \longrightarrow S \longrightarrow \pi^{\ast} (\bbC \oplus L) \longrightarrow Q \longrightarrow 0.
	\end{equation}
	Using the definition of $\omega$ and the above exact sequence,
	we obtain
	\begin{equation}
		c_{1}(S) = [\omega] - \pi^{\ast} [\omega_{0}],
		\qquad
		c_{1}(Q) = - [\omega] + \pi^{\ast} [\omega_{1}].
	\end{equation}
	Moreover,
	we have
	\begin{equation}
		[\omega] \cup ([\omega] - \pi^{\ast} [\omega_{0} + \omega_{1}])
		= - c_{1}(S) \cup c_{1}(Q)
		= - c_{2}(\pi^{\ast} (\bbC \oplus L))
		= 0.
	\end{equation}
	Set $\alpha \coloneqq \omega - \pi^{\ast} (\omega_{0} + \omega_{1})$,
	which is a closed real $(1, 1)$-form on $Y$ satisfying $[\omega] \cup [\alpha] = 0$.
	The relative holomorphic tangent bundle $T^{1, 0}_{Y / B}$ of $\pi \colon Y \to B$
	is isomorphic to $\Hom(S, Q) \cong S^{\ast} \otimes Q$.
	Since $c(T^{1, 0} B) = 1$,
	we have
	\begin{equation}
		c(T^{1, 0} Y)
		= c(T^{1, 0}_{Y / B})
		= 1 - 2 [\omega] + \pi^{\ast} [\omega_{0} + \omega_{1}].
	\end{equation}
	Using this formula and integrating along the fibers, we obtain
	\begin{align}
		\int_{Y} [\alpha] \cup c_{1}(T^{1, 0} Y)^{k}
		&= \int_{Y} ([\omega] - \pi^{\ast} [\omega_{0} + \omega_{1}])
			\cup (- 2 [\omega] + \pi^{\ast} [\omega_{0} + \omega_{1}])^{k} \\
		&= \int_{Y} [\omega] \cup \pi^{\ast} [\omega_{0} + \omega_{1}]^{k} \\
		&= - \int_{B} [\omega_{0} + \omega_{1}]^{k} \\
		&= - \int_{B} (\omega_{0}^{k} + \omega_{1}^{k}).
	\end{align}
	Since $\omega_{t}$ is a continuous family of pseudo-\Kahler forms,
	the integrals of $\omega_{0}^{k}$ and $\omega_{1}^{k}$ are non-zero and have the same sign.
	Consequently,
	\begin{equation}
		\int_{Y} [\alpha] \cup c_{1}(T^{1, 0} Y)^{k} \neq 0,
	\end{equation}
	which completes the proof.
\end{proof}

\section{Low-dimensional building blocks}
\label{section:building-blocks}

In this section, we construct closed pseudo-\Kahler manifolds $(Y_{d}, \omega_{d})$
of complex dimensions $d = 2, 3, 4$
and closed real $(1, 1)$-forms $\alpha_{d}$ on $Y_{d}$ such that
\begin{equation}
	[\omega_{d}] \in H^{2}(Y_{d}; \bbZ),
	\qquad
	[\omega_{d}] \cup [\alpha_{d}] = 0,
	\qquad
	A_{d} \coloneqq \int_{Y_{d}} [\alpha_{d}] \cup c_{1}(T^{1, 0} Y_{d})^{d - 1} \neq 0.
\end{equation}
We also record the Chern classes needed for the product construction in \cref{section:Chern-classes}.

\subsection{A product of projective lines}

Let $\FS$ be the Fubini-Study form on $\cps^{1}$ normalized by $[\FS] = c_{1}(\calO(1))$.
Define $Y_{2} \coloneqq \cps^{1} \times \cps^{1}$.
Write $\pr_{j} \colon Y_{2} \to \cps^{1}$ for the $j$-th projection.
The total Chern class $c(T^{1, 0} Y_{2})$ of $T^{1, 0} Y_{2}$ is given by
\begin{align}
	c(T^{1, 0} Y_{2})
	&= (1 + \pr_{1}^{\ast} c_{1}(T^{1, 0} \cps^{1})) \cup (1 + \pr_{2}^{\ast} c_{1}(T^{1, 0} \cps^{1})) \\
	&= 1 + 2 (\pr_{1}^{\ast} [\FS] + \pr_{2}^{\ast} [\FS]) + 4 \pr_{1}^{\ast} [\FS] \cup \pr_{2}^{\ast} [\FS].
\end{align}
Set
\begin{equation}
	\omega_{2}
	\coloneqq \pr_{1}^{\ast} \FS - \pr_{2}^{\ast} \FS,
	\qquad
	\alpha_{2}
	\coloneqq \pr_{1}^{\ast} \FS + \pr_{2}^{\ast} \FS.
\end{equation}
The closed real $(1, 1)$-form $\omega_{2}$ is an integral pseudo-\Kahler form of signature $(1, 1)$
and satisfies $\omega_{2} \wedge \alpha_{2} = 0$.
Moreover,
\begin{equation}
	A_{2}
	\coloneqq \int_{Y_{2}} [\alpha_{2}] \cup c_{1}(T^{1, 0} Y_{2})
	= 2 \int_{Y_{2}} (\pr_{1}^{\ast} [\FS] + \pr_{2}^{\ast} [\FS])^{2}
	= 4
	\neq 0.
\end{equation}

\subsection{A threefold over a complex torus}

Define $\Sigma \coloneqq \bbC / (\bbZ + \sqrt{- 1} \bbZ)$
and $B \coloneqq \Sigma \times \Sigma$.
The closed complex surface $B$ satisfies $c(T^{1, 0} B) = 1$.
We use the induced local holomorphic coordinates $(z, w)$.
Then
\begin{equation}
	\omega_{0} \coloneqq \sqrt{- 1} (d z \wedge d \ovz - d w \wedge d \ovw),
	\qquad
	\omega_{1} \coloneqq \sqrt{- 1} (d z \wedge d \ovw + d w \wedge d \ovz)
\end{equation}
are closed real $(1, 1)$-forms on $B$ satisfying
$[\omega_{0}], [\omega_{1}] \in H^{2}(B; \bbZ)$ and $\omega_{0} \wedge \omega_{1} = 0$.
Moreover,
$\omega_{t} = (1 - t) \omega_{0} + t \omega_{1}$ satisfies
\begin{equation}
	\omega_{t}^{2}
	= (1 - t)^{2} \omega_{0}^{2} + t^{2} \omega_{1}^{2}
	= 2 ((1 - t)^{2} + t^{2}) d z \wedge d \ovz \wedge d w \wedge d \ovw
	\neq 0.
\end{equation}
Hence $\omega_{t}$ is a pseudo-\Kahler form of signature $(1, 1)$ for every $t \in \clcl{0}{1}$.
\cref{prop:projective-bundle} implies that
there exist a closed pseudo-\Kahler threefold $(Y_{3}, \omega_{3})$ of signature $(1, 2)$
and a closed real $(1, 1)$-form $\alpha_{3}$ on $Y_{3}$ such that
$c(T^{1, 0} Y_{3}) = 1 + c_{1}(T^{1, 0} Y_{3})$,
$[\omega_{3}] \in H^{2}(Y_{3}; \bbZ)$,
$[\omega_{3}] \cup [\alpha_{3}] = 0$,
and
\begin{equation}
	A_{3} \coloneqq \int_{Y_{3}} [\alpha_{3}] \cup c_{1}(T^{1, 0} Y_{3})^{2} \neq 0.
\end{equation}

A related symplectic construction on an $S^{2}$-bundle over the four-torus
appears in Cho~\cite{Cho2016}*{Example 5.3}.
Both constructions use endpoint classes with vanishing cup product;
see also \cite{Cho2016}*{Proposition 1.7}.
Here we construct an integral pseudo-\Kahler form
and establish a non-zero Chern-class pairing for use in the subsequent CR construction.

\subsection{A fourfold over a nilmanifold}

Let $\frakg$ be a real six-dimensional Lie algebra with basis $e_{1}, \dots , e_{6}$
whose only non-zero brackets are
\begin{equation}
	\comm{e_{1}}{e_{2}} = - e_{4},
	\qquad
	\comm{e_{1}}{e_{3}} = - e_{5},
	\qquad
	\comm{e_{2}}{e_{3}} = - e_{6}.
\end{equation}
The Lie algebra $\frakg$ is two-step nilpotent.
This is the Lie algebra $\frakh_{7} = (0,0,0,12,13,23)$,
whose pseudo-\Kahler structures were studied by Cordero, Fern\'andez, and Ugarte
\cite{Cordero-Fernandez-Ugarte2004}*{Proposition 3.9 (i)}.
Let $G$ be the simply connected Lie group corresponding to $\frakg$.
The exponential map
\begin{equation}
	\exp \colon \frakg \longrightarrow G
\end{equation}
is a diffeomorphism;
we write $\log \colon G \to \frakg$ for the inverse map.
The Baker--Campbell--Hausdorff formula states that
\begin{equation}
	\log (\exp X \cdot \exp Y)
	= X + Y + \frac{1}{2} \comm{X}{Y}
\end{equation}
for any $X, Y \in \frakg$.
Consider the lattice $\Lambda$ in $\frakg$ defined by
\begin{equation}
	\Lambda
	\coloneqq 2 \sqrt{3} \bbZ e_{1} \oplus 2 \bbZ e_{2} \oplus 2 \sqrt{3} \bbZ e_{3}
		\oplus \sqrt{3} \bbZ e_{4} \oplus \bbZ e_{5} \oplus \sqrt{3} \bbZ e_{6}.
\end{equation}
This lattice satisfies $\comm{\Lambda}{\Lambda} \subset 2 \Lambda$.
Hence $\Gamma \coloneqq \exp \Lambda$ is a lattice in $G$,
and $B \coloneqq \Gamma \backslash G$ is a closed real six-manifold.

Consider a left-invariant almost complex structure on $G$ defined by
\begin{equation}
	J e_{1} = e_{2},
	\qquad
	J e_{3} = e_{4},
	\qquad
	J e_{5} = e_{6}.
\end{equation}
This is the complex structure described in \cite{Fino-Rollenske-Ruppenthal2019}*{Proposition 5.3 (ii)}.
A direct calculation shows that $J$ is integrable,
and so $G$ has a left-invariant complex structure.
Hence $B$ is a closed complex threefold.
Since $T^{1, 0} B$ is trivial as a smooth complex vector bundle,
we have $c(T^{1, 0} B) = 1$.
Let $e^{1}, \dots , e^{6}$ be the dual basis of $\frakg^{\ast}$,
whose elements we regard as left-invariant one-forms on $G$.
The structure equations of $\frakg$ imply
\begin{equation}
	d e^{1} = d e^{2} = d e^{3} = 0,
	\qquad
	d e^{4} = e^{1 2},
	\qquad
	d e^{5} = e^{1 3},
	\qquad
	d e^{6} = e^{2 3};
\end{equation}
here and in the sequel,
we adopt the abbreviation $e^{i j k \dotsm}$ for $e^{i} \wedge e^{j} \wedge e^{k} \wedge \dotsm$.
Define the real $(1, 1)$-forms
\begin{equation}
	\mu \coloneqq \sqrt{3} e^{1 5} + \sqrt{3} e^{2 6},
	\qquad
	\nu \coloneqq - e^{1 6} + e^{2 5} + 2 e^{3 4}.
\end{equation}
A direct calculation shows that $\mu$ and $\nu$ are $d$-closed
and that
\begin{gather}
	\mu^{2} = - 6 e^{1 2 5 6},
	\qquad
	\nu^{2} = - 2 e^{1 2 5 6} - 4 e^{1 3 4 6} + 4 e^{2 3 4 5}, \\
	\mu^{3} = \mu \wedge \nu^{2} = 0,
	\qquad
	\mu^{2} \wedge \nu = \nu^{3} = - 12 e^{1 2 3 4 5 6}.
\end{gather}

The rational structure $\frakg_{\bbQ}$ of $\frakg$ determined by $\Gamma$ is given by
\begin{equation}
	\frakg_{\bbQ}
	= \Span_{\bbQ} (\log \Gamma)
	= \Lambda \otimes \bbQ.
\end{equation}
By the rational version of Nomizu's theorem~\cite{Lambe-Priddy1982}*{Corollary 4.2},
the Chevalley-Eilenberg complex
\begin{equation}
	\rbra*{\bigwedge \frakg_{\bbQ}^{\ast}, d_{\mathrm{CE}}}
\end{equation}
computes $H^{\ast}(B; \bbQ)$,
and this identification is compatible with extension of scalars to $\bbR$
and the de Rham isomorphism.
Since $\mu$ and $\nu$ are closed invariant forms with rational coefficients relative to $\frakg_{\bbQ}$,
their de Rham classes are rational.
Hence there exists a positive integer $N$ such that
both $N [\mu]$ and $N [\nu]$ lie in the image of integral cohomology.
Set
\begin{equation}
	\omega_{0} \coloneqq N (- \mu + \nu),
	\qquad
	\omega_{1} \coloneqq N (\mu + \nu).
\end{equation}
Then $[\omega_{0}], [\omega_{1}] \in H^{2}(B; \bbZ)$ and
\begin{equation}
	\omega_{0} \wedge \omega_{1}
	= N^{2} (- \mu^{2} + \nu^{2})
	= 4 N^{2} (e^{1 2 5 6} - e^{1 3 4 6} + e^{2 3 4 5})
	= 4 N^{2} \, d(e^{4 5 6}).
\end{equation}
In particular,
$[\omega_{0}] \cup [\omega_{1}] = 0$.
Moreover,
\begin{equation}
	\omega_{t}
	\coloneqq (1 - t) \omega_{0} + t \omega_{1}
	= N ((2 t - 1) \mu + \nu)
\end{equation}
satisfies
\begin{equation}
	\omega_{t}^{3}
	= 3 N^{3} (2 t - 1)^{2} \mu^{2} \wedge \nu + N^{3} \nu^{3}
	= - 12 N^{3} (3 (2 t - 1)^{2} + 1) e^{1 2 3 4 5 6}
	\neq 0
\end{equation}
and
\begin{equation}
	\omega_{t}(e_{3}, J e_{3})
	= \omega_{t}(e_{3}, e_{4})
	= 2 N
	> 0.
\end{equation}
Hence $\omega_{t}$ is a pseudo-\Kahler form of signature $(2, 1)$ for every $t \in \clcl{0}{1}$.
\cref{prop:projective-bundle} implies that
there exist a closed pseudo-\Kahler fourfold $(Y_{4}, \omega_{4})$ of signature $(2, 2)$
and a closed real $(1, 1)$-form $\alpha_{4}$ on $Y_{4}$ such that
$c(T^{1, 0} Y_{4}) = 1 + c_{1}(T^{1, 0} Y_{4})$,
$[\omega_{4}] \in H^{2}(Y_{4}; \bbZ)$,
$[\omega_{4}] \cup [\alpha_{4}] = 0$,
and
\begin{equation}
	A_{4} \coloneqq \int_{Y_{4}} [\alpha_{4}] \cup c_{1}(T^{1, 0} Y_{4})^{3} \neq 0.
\end{equation}

\section{Products and non-vanishing of Chern monomials}
\label{section:Chern-classes}

We now combine the building blocks from \cref{section:building-blocks}
to prove \cref{thm:non-trivial-higher-Chern-class}.

\subsection{The product construction}
\label{subsection:product-construction}

Fix $n$ and $m$ as in \cref{thm:non-trivial-higher-Chern-class}.
We first note that $2 m \leq n \leq 4 m$.
Write $n = (q + 1) m + s$,
where $1 \leq q \leq 3$ and $0 \leq s < m$,
and set
\begin{equation}
	n_{j}
	\coloneqq
	\begin{cases}
		q + 2 & (1 \leq j \leq s), \\
		q + 1 & (s + 1 \leq j \leq m).
	\end{cases}
\end{equation}
These integers satisfy $n_{1} + \dots + n_{m} = n$.
Note that $q = 3$ implies $s = 0$,
since $n \leq 4 m$.
Thus $2 \leq n_{j} \leq 4$ for every $j$,
so all the factors $Y_{n_{j}}$ used below have already been constructed.
Consider the $n$-dimensional closed complex manifold
\begin{equation}
	Y \coloneqq Y_{n_{1}} \times \cdots \times Y_{n_{m}}
\end{equation}
and denote by $\pr_{j} \colon Y \to Y_{n_{j}}$ the $j$-th projection.
Let
\begin{equation}
	J_{3} \coloneqq \Set{j \in \{1, \dots , m\} | n_{j} = 3},
	\qquad
	b \coloneqq \# J_{3}.
\end{equation}
Choose a subset $J_{-} \subset J_{3}$ with $\# J_{-} = \lfloor b / 2 \rfloor$,
and set
\begin{equation}
	\varepsilon_{j}
	\coloneqq
	\begin{cases}
		- 1 & (j \in J_{-}), \\
		1 & (j \in \{1, \dots , m\} \setminus J_{-}).
	\end{cases}
\end{equation}
Define
\begin{equation}
	\omega
	\coloneqq \sum_{j = 1}^{m} \varepsilon_{j} \pr_{j}^{\ast} \omega_{n_{j}},
	\qquad
	\alpha
	\coloneqq \pr_{1}^{\ast} \alpha_{n_{1}} \wedge \dots \wedge \pr_{m}^{\ast} \alpha_{n_{m}}.
\end{equation}
Each $Y_{3}$ factor indexed by $J_{-}$ contributes signature $(2, 1)$,
whereas each remaining $Y_{3}$ factor contributes signature $(1, 2)$.
The $Y_{2}$ and $Y_{4}$ factors have signatures $(1, 1)$ and $(2, 2)$, respectively.
Hence the difference between the negative and positive indices of $\omega$ is
\begin{equation}
	b - 2 \lfloor b / 2 \rfloor.
\end{equation}
Since $n \equiv b \pmod{2}$,
the form $\omega$ has signature $(\lfloor n / 2 \rfloor, \lceil n / 2 \rceil)$.
Moreover,
$\omega$ is a closed real $(1, 1)$-form satisfying
$[\omega] \in H^{2}(Y; \bbZ)$ and $[\omega] \cup [\alpha] = 0$.
Thus $(Y, \omega)$ is a pseudo-\Kahler manifold with the required signature.

\subsection{Proof of the main theorem}

For the manifold and forms constructed above, we prove the following non-vanishing statement.

\begin{theorem}
\label{thm:higher-Chern-violates-hard-Lefschetz}
	Let $n$ and $m$ be as in \cref{thm:non-trivial-higher-Chern-class}.
	There exist an $n$-dimensional closed pseudo-\Kahler manifold $(Y, \omega)$
	of signature $(\lfloor n / 2 \rfloor, \lceil n / 2 \rceil)$
	and a closed real $(m, m)$-form $\alpha$ on $Y$ such that
	$[\omega] \in H^{2}(Y; \bbZ)$,
	$[\omega] \cup [\alpha] = 0$,
	and
	\begin{equation}
		\int_{Y} [\alpha] \cup c_{r_{1}}(T^{1, 0} Y) \cup \dotsm \cup c_{r_{k}}(T^{1, 0} Y)
			\cup c_{1}(T^{1, 0} Y)^{n - m - \abs{r}} \neq 0
	\end{equation}
	for any integer $k \geq 1$ and any $(r_{1}, \dots , r_{k}) \in \{1, \dots , m\}^{k}$
	satisfying $\abs{r} \leq n - m$.
\end{theorem}

\begin{proof}
	Let $Y$, $\omega$, and $\alpha$ be as in \cref{subsection:product-construction}.
	Let $r = (r_{1}, \dots , r_{k}) \in \{1, \dots , m\}^{k}$ satisfy $\abs{r} \leq n - m$.
	Set $h \coloneqq n - m - \abs{r}$,
	and define
	\begin{equation}
		(\wtr_{1}, \dots , \wtr_{k}, \wtr_{k + 1}, \dots , \wtr_{k + h})
		\coloneqq (r_{1}, \dots , r_{k}, \underbrace{1, \dots , 1}_{\text{$h$ entries}}).
	\end{equation}
	All entries of this multi-index lie in $\{1, \dots, m\}$,
	and its total weight is $n - m$.
	Moreover, the integral associated with $\wtr$ is precisely
	\begin{equation}
		\int_{Y} [\alpha] \cup c_{r_{1}}(T^{1, 0} Y) \cup \dotsm \cup c_{r_{k}}(T^{1, 0} Y)
			\cup c_{1}(T^{1, 0} Y)^{h}.
	\end{equation}
	Thus the assertion for $r$ follows from that for $\wtr$.
	In the remainder of the proof, we therefore assume that $\abs{r} = n - m$.
	It remains to show that the integral
	\begin{equation}
	\label{eq:non-trivial-integral}
		\int_{Y} [\alpha] \cup c_{r_{1}}(T^{1, 0} Y) \cup \dotsm \cup c_{r_{k}}(T^{1, 0} Y)
	\end{equation}
	is non-zero.
	To this end,
	consider the Chern polynomial
	\begin{equation}
		c_{t}(T^{1, 0} Y)
		= \sum_{l \geq 0} c_{l}(T^{1, 0} Y) t^{l}
		= \prod_{j = 1}^{m} \pr_{j}^{\ast} c_{t}(T^{1, 0} Y_{n_{j}})
	\end{equation}
	and the $i$-th elementary symmetric polynomial $e_{i}(t_{1}, \dots , t_{k})$.
	The computations in the previous section give
	\begin{gather}
		\int_{Y_{2}} [\alpha_{2}] \cup \prod_{l = 1}^{k} c_{t_{l}}(T^{1, 0} Y_{2})
		= A_{2} e_{1}(t_{1}, \dots , t_{k}), \\
		\int_{Y_{3}} [\alpha_{3}] \cup \prod_{l = 1}^{k} c_{t_{l}}(T^{1, 0} Y_{3})
		= A_{3} e_{2}(t_{1}, \dots , t_{k}), \\
		\int_{Y_{4}} [\alpha_{4}] \cup \prod_{l = 1}^{k} c_{t_{l}}(T^{1, 0} Y_{4})
		= A_{4} e_{3}(t_{1}, \dots , t_{k}).
	\end{gather}
	Thus we have
	\begin{equation}
		\int_{Y} [\alpha] \cup \prod_{l = 1}^{k} c_{t_{l}}(T^{1, 0} Y)
		= \rbra*{\prod_{j = 1}^{m} A_{n_{j}}} \prod_{j = 1}^{m} e_{n_{j} - 1}(t_{1}, \dots, t_{k}).
	\end{equation}
	Moreover,
	the coefficient of $t_{1}^{r_{1}} \dotsm t_{k}^{r_{k}}$ in the left-hand side
	is equal to \eqref{eq:non-trivial-integral}.
	Hence it suffices to show that
	the coefficient of $t_{1}^{r_{1}} \dotsm t_{k}^{r_{k}}$ in
	\begin{equation}
		\prod_{j = 1}^{m} e_{n_{j} - 1}(t_{1}, \dots, t_{k})
	\end{equation}
	is non-zero.
	
	Since all coefficients of the elementary symmetric polynomials are nonnegative integers,
	it suffices to choose one monomial from each factor
	$e_{n_{j} - 1}(t_{1}, \dots , t_{k})$
	so that the product of the chosen monomials is
	$t_{1}^{r_{1}} \dotsm t_{k}^{r_{k}}$.
	We construct such a choice as follows.
	Consider the sequence
	\begin{equation}
		\underbrace{1, \dots , m, 1, \dots , m, \dots , 1, \dots , m}_{\text{$q$ times}},
		1, \dots , s.
	\end{equation}
	Here the final segment $1, \dots , s$ is understood to be empty when $s = 0$.
	Each of the labels $1, \dots , s$ occurs $q + 1$ times,
	whereas each of the labels $s + 1, \dots , m$ occurs $q$ times.
	By the definition of $n_{j}$,
	the label $j$ therefore occurs exactly $n_{j} - 1$ times.
	Moreover, the total length of the sequence is
	\begin{equation}
		q m + s = n - m = r_{1} + \dots + r_{k}.
	\end{equation}
	We can therefore divide this sequence into $k$ consecutive blocks of lengths $r_{1}, \dots , r_{k}$.
	Since $r_{l} \leq m$,
	no label occurs twice in the same block.
	For each $1 \leq j \leq m$,
	let $I_{j} \subset \{1, \dots , k\}$ be
	the set of indices of the blocks containing the label $j$.
	Then $\# I_{j} = n_{j} - 1$,
	and each index $l$ belongs to exactly $r_{l}$ of the sets $I_{j}$.
	Thus $\prod_{l \in I_{j}} t_{l}$ is a monomial appearing in $e_{n_{j} - 1}(t_{1}, \dots , t_{k})$,
	and
	\begin{equation}
		\prod_{j = 1}^{m} \rbra*{\prod_{l \in I_{j}} t_{l}}
		= t_{1}^{r_{1}} \dotsm t_{k}^{r_{k}}.
	\end{equation}
	Therefore,
	the desired coefficient is positive.
	As $A_{n_{j}} \neq 0$ for every $j$,
	the integral \eqref{eq:non-trivial-integral} is non-zero.
	This completes the proof.
\end{proof}

\begin{proof}[Proof of \cref{thm:non-trivial-higher-Chern-class}]
	Take $(Y, \omega)$ and $\alpha$ as in \cref{thm:higher-Chern-violates-hard-Lefschetz},
	and let $p \colon S \to Y$ be the associated circle bundle.
	By \cref{section:circle-bundles},
	$S$ is a closed Levi non-degenerate CR manifold of dimension $2 n + 1$
	and signature $(\lfloor n / 2 \rfloor, \lceil n / 2 \rceil)$.
	For any admissible multi-index $r$, put
	\begin{equation}
		\beta_{r}
		\coloneqq [\alpha] \cup c_{1}(T^{1, 0} Y)^{n - m - \abs{r}}
		\in H^{2 n - 2 \abs{r}}(Y; \bbR).
	\end{equation}
	Then $[\omega] \cup \beta_{r} = 0$,
	and \cref{thm:higher-Chern-violates-hard-Lefschetz} gives
	\begin{equation}
		\int_{Y} \beta_{r} \cup c_{r_{1}}(T^{1, 0} Y) \cup \dotsm \cup c_{r_{k}}(T^{1, 0} Y) \neq 0.
	\end{equation}
	Applying \cref{lem:chern-non-vanishing-criterion} yields
	\begin{equation}
		c_{r_{1}}(T^{1, 0} S) \cup \dotsm \cup c_{r_{k}}(T^{1, 0} S)
		= p^{\ast} \rbra*{c_{r_{1}}(T^{1, 0} Y) \cup \dotsm \cup c_{r_{k}}(T^{1, 0} Y)}
		\neq 0,
	\end{equation}
	which completes the proof.
\end{proof}

\subsection{Further consequences and limitations}

The pseudo-\Kahler counterpart of \cref{thm:non-trivial-first-Chern-class}
is also a direct consequence of \cref{thm:higher-Chern-violates-hard-Lefschetz}.

\begin{corollary}
\label{thm:first-Chern-violates-hard-Lefschetz}
	For each $n \geq 2$,
	there exist an $n$-dimensional closed pseudo-\Kahler manifold $(Y, \omega)$
	of signature $(\lfloor n / 2 \rfloor, \lceil n / 2 \rceil)$
	and a closed real $(\lceil n / 4 \rceil, \lceil n / 4 \rceil)$-form $\alpha$ on $Y$ such that
	$[\omega] \in H^{2}(Y; \bbZ)$,
	$[\omega] \cup [\alpha] = 0$,
	and
	\begin{equation}
		\int_{Y} [\alpha] \cup c_{1}(T^{1, 0} Y)^{\lfloor 3 n / 4 \rfloor} \neq 0.
	\end{equation}
	In particular,
	\begin{equation}
		c_{1}(T^{1, 0} Y)^{\lfloor 3 n / 4 \rfloor}
		\notin \Im \rbra*{[\omega] \cup \colon H^{2 \lfloor 3 n / 4 \rfloor - 2}(Y; \bbR)
			\longrightarrow H^{2 \lfloor 3 n / 4 \rfloor}(Y; \bbR)}.
	\end{equation}
\end{corollary}

\begin{proof}
	Apply \cref{thm:higher-Chern-violates-hard-Lefschetz}
	with $m = \lceil n / 4 \rceil$ and $r_{1} = \dots = r_{k} = 1$,
	where $k = n - m = \lfloor 3 n / 4 \rfloor$.
	The last assertion follows from $[\omega] \cup [\alpha] = 0$
	and the non-zero pairing with $[\alpha]$.
\end{proof}

\begin{remark}
	For a closed symplectic manifold $(Y, \omega)$ of real dimension $2 n$,
	let $L_{\omega}$ denote cup product with $[\omega]$.
	The pair $(Y, \omega)$ is said to satisfy the \emph{hard Lefschetz property} if
	\begin{equation}
		L_{\omega}^{n - \ell} \colon H^{\ell}(Y; \bbR) \longrightarrow H^{2 n - \ell}(Y; \bbR)
	\end{equation}
	is an isomorphism for every integer $0 \leq \ell \leq n$.
	If $\omega$ is a \Kahler form,
	$(Y, \omega)$ satisfies the hard Lefschetz property.
	However,
	the hard Lefschetz property need not hold for pseudo-\Kahler manifolds;
	see \cite{Yamada2005}*{Example 5.2} for example.
	Our construction gives integral pseudo-\Kahler manifolds
	that fail the hard Lefschetz property in every complex dimension $n \geq 3$.
	In these examples,
	the failure is detected by Chern monomials
	lying outside the image of cup product with $[\omega]$ as follows.
	
	Suppose that $2 m < n$ in \cref{thm:higher-Chern-violates-hard-Lefschetz}.
	For every multi-index $r = (r_{1}, \dots , r_{k}) \in \{1, \dots , m\}^{k}$ of total weight $n - m$,
	the non-zero pairing with $[\alpha]$ and the relation $L_{\omega} [\alpha] = 0$ imply
	\begin{equation}
		c_{r_{1}}(T^{1, 0} Y) \cup \dotsm \cup c_{r_{k}}(T^{1, 0} Y)
		\notin \Im \rbra*{L_{\omega} \colon H^{2 (n - m) - 2}(Y; \bbR) \longrightarrow H^{2 (n - m)}(Y; \bbR)}.
	\end{equation}
	Since $2 (n - m) > n$, the hard Lefschetz property would force this map to be surjective:
	the map $L_{\omega}^{n - 2 m} \colon H^{2 m}(Y; \bbR) \longrightarrow H^{2 (n - m)}(Y; \bbR)$
	factors through it.
	In particular, taking $m = \lceil n / 4 \rceil$ gives examples failing the hard Lefschetz property
	in every complex dimension $n \geq 3$,
	with the failure detected by $c_{1}(T^{1, 0} Y)^{\lfloor 3 n / 4 \rfloor}$.
\end{remark}

\begin{remark}
	For $n = 2, 3$,
	the circle-bundle construction above cannot yield a counterexample
	to the vanishing theorem for Chern monomials.
	Indeed,
	the condition $2 \abs{r} \geq n + 2$ implies $\abs{r} \geq n$.
	Since $[\omega]^{n} \neq 0$,
	the Lefschetz map
	\begin{equation}
		[\omega] \cup \colon H^{2 n - 2}(Y; \bbR) \longrightarrow H^{2 n}(Y; \bbR)
	\end{equation}
	is surjective.
	The Gysin exact sequence therefore implies that
	\begin{equation}
		p^{\ast} \colon H^{2 n}(Y; \bbR) \longrightarrow H^{2 n}(S; \bbR)
	\end{equation}
	is zero.
	Thus Chern monomials of weight $\abs{r} = n$ pull back to zero on $S$,
	while those of weight $\abs{r} > n$ vanish on $Y$ for dimensional reasons.
	
	For $n = 5$,
	the construction above yields $c_{1}(T^{1, 0} S)^{3} \neq 0$,
	whereas a counterexample to the vanishing theorem would require
	a non-zero Chern monomial of weight at least $4$.
	Hence the case $n = 5$ is not settled by the present construction.
\end{remark}

\section*{Acknowledgements}
The author is grateful to Shuho Kanda for helpful comments.

\section*{Declaration of generative AI use}
During the preparation of this manuscript,
the author used ChatGPT and Codex (OpenAI) to assist with English-language editing,
manuscript organization, literature searches, and the development and review of mathematical arguments.
The author critically evaluated and revised all AI-assisted material incorporated into the manuscript,
independently verified the mathematical statements, proofs, and cited sources,
and takes full responsibility for the accuracy and integrity of the final manuscript.

\bibliography{my-reference,my-reference-preprint}

\end{document}